\documentclass[11pt]{article}
\usepackage[utf8]{inputenc}
\usepackage{amssymb}
\usepackage{amsmath}
\usepackage{amsthm}
\usepackage{siunitx}
\usepackage{comment}
\usepackage{indentfirst}
\usepackage{color,soul}
\usepackage{longtable}

\usepackage[citation-order]{amsrefs}
\newenvironment{rezabib}
  {\bibdiv\biblist\setupbib}
  {\endbiblist\endbibdiv}

\def\setupbib{\catcode`@=\active}
\begingroup\lccode`~=`@
  \lowercase{\endgroup\def~}#1#{\gatherkey{#1}}
\def\gatherkey#1#2{\gatherkeyaux{#1}#2\gatherkeyaux}
\def\gatherkeyaux#1#2,#3\gatherkeyaux{\bib{#2}{#1}{#3}}

\makeatletter
\def\bbordermatrix#1{\begingroup \m@th
  \@tempdima 4.75\p@
  \setbox\z@\vbox{%
    \def\cr{\crcr\noalign{\kern2\p@\global\let\cr\endline}}%
    \ialign{$##$\hfil\kern2\p@\kern\@tempdima&\thinspace\hfil$##$\hfil
      &&\quad\hfil$##$\hfil\crcr
      \omit\strut\hfil\crcr\noalign{\kern-\baselineskip}%
      #1\crcr\omit\strut\cr}}%
  \setbox\tw@\vbox{\unvcopy\z@\global\setbox\@ne\lastbox}%
  \setbox\tw@\hbox{\unhbox\@ne\unskip\global\setbox\@ne\lastbox}%
  \setbox\tw@\hbox{$\kern\wd\@ne\kern-\@tempdima\left[\kern-\wd\@ne
    \global\setbox\@ne\vbox{\box\@ne\kern2\p@}%
    \vcenter{\kern-\ht\@ne\unvbox\z@\kern-\baselineskip}\,\right]$}%
  \null\;\vbox{\kern\ht\@ne\box\tw@}\endgroup}
\makeatother
\newtheorem*{corollary*}{Corollary}
\newtheorem*{theorem*}{Theorem}

\theoremstyle{definition}
\newtheorem{theorem}{Theorem}

\newtheorem{definition}[theorem]{Definition}

\theoremstyle{remark}

\theoremstyle{remark}
\newtheorem{example}{Example}

\theoremstyle{remark}
\newtheorem*{example*}{Example}

\theoremstyle{remark}
\newtheorem*{remark}{Remark}

\theoremstyle{remark}

\newtheorem{prop}[theorem]{Proposition}

\newcommand{\D}{\mathcal{D}}

\newcommand{\R}{\mathcal{R}}

\def\0{{\bm 0}}   

\title{Balanced Weighing Matrices}
\author{
 Hadi Kharaghani\thanks{Department of Mathematics and Computer Science, University of Lethbridge,
Lethbridge, Alberta, T1K 3M4, Canada. \texttt{kharaghani@uleth.ca}}
\and
Thomas Pender\thanks{Department of Mathematics and Computer Science, University of Lethbridge,
Lethbridge, Alberta, T1K 3M4, Canada. \texttt{Thomas.pender@uleth.ca}}
\and
  Sho Suda\thanks{Department of Mathematics,  National Defense Academy of Japan, Yokosuka, Kanagawa 239-8686, Japan. \texttt{ssuda@nda.ac.jp}}
}
\date{\today}

\begin{document}
\maketitle

\begin{abstract}
A unified approach to the construction of weighing matrices and certain symmetric designs is presented.
Assuming the weight $p$ in a weighing matrix $W(n,p)$ is a prime power, it is shown that there is a $$W\left(\frac{p^{m+1}-1}{p-1}(n-1)+1,p^{m+1}\right)$$ for each positive integer $m$. The case of $n=p+1$ reduces to the balanced weighing matrices with classical parameters $$W\left(\frac{p^{m+2}-1}{p-1},p^{m+1}\right).$$ The equivalence with certain classes of association schemes is discussed in detail.
\end{abstract}

\section{Introduction}

A \emph{weighing matrix} of order $n$ and weight $p$, denoted $W(n,p)$, is a $(0,\pm 1)$-matrix $W$ of order $n$ such that $WW^t = pI_n$. The special cases in which $n=p+1$ is called a \emph{conference} matrix and $n=p$ is a \emph{Hadamard} matrix. A weighing matrix $W(n,p)$ is {said} to be \emph{balanced} if, upon setting each non-zero entry to unity, {it} provides the incidence matrix of a symmetric $(n,p,\lambda)$ balanced incomplete block design. 

Balanced  weighing matrices with {\it classical parameters} include weighing matrices 
$W\left(\frac{p^{m+2}-1}{p-1},p^{m+1}\right)$,  for each positive integer $m$. Upon squaring all the entries it reduces to the symmetric $$BIBD\left(\frac{p^{m+2}-1}{p-1},p^{m+1},p^{m+1}-p^m\right).$$ 

Most of the existing constructive methods for weighing matrices relate to orders which are multiples of four. Indeed,  it is conjectured that there is a $W(4n,k)$  for positive integers $n$ and $1\le k \le 4n$, see \cite{Seb}. In contrast to that case, the progress for the remaining cases of $n\equiv 1,2,3\pmod 4$ have been quite limited. 
 
In this paper,  a new method is presented that is suitable for the remaining three cases. The unifying method of construction establishes a link between the construction and applications of weighing matrices,  and the construction of certain symmetric designs. The outcome is the existence of new weighing matrices,  and a new method of construction for balanced weighing matrices and balanced generalized weighing matrices.

The basic assumption on which most of the results {found} herein depend is that the weight $p$ in the $W(n,p)$ is a prime power. With the restricted value of the weight $p$,  many powerful tools such as balanced generalized weighing  matrices and  orthogonal arrays  will be available and instrumental in the development of  some new techniques presented in this paper. 

The main references for the paper are \cite{se-yam,Seb}. 
\section{Preliminaries}
\subsection{BGWs over $\mathbb{Z}_n$}
Let $G$ be a multiplicatively written group, and let $W=[w_{ij}]$ be a ($0,{G}$)-matrix of order $v$. The matrix $W$ is a \emph{balanced generalized weighing matrix} over $G$ with parameters $(v,k,\lambda)$, denoted $BGW(v,k,\lambda;G)$, if each row of $W$ contains exactly $k$ nonzero entries, and if the multiset $\{w_{ik}w_{jk}^{-1}\mid 1\leq k\leq v, w_{ik} \neq 0 \neq w_{j,k}\}$ contains exactly $\lambda/|G|$ copies of every element of $G$, for any distinct $i,j\in\{1,\ldots,v\}$. 

\begin{example}
{A  balanced weighing matrix of order $v$ and weight $k$ is a \\$BGW(v,k,\lambda;\{1,-1\})$.} 
\end{example}

\begin{example}
Every conference matrix of order $v$ is a $BGW(v,v-1,v-2;\{1,-1\})$. 
\end{example}

\begin{example}
For the prime power $q$, the $BGW(\frac{q^{m+1}-1}{q-1},q^m,q^{m-1}(q-1))$ over the cyclic group $G$ whose order divides $q-1$, for each positive integer $m$, is the largest class of known balanced generalized weighing matrices with the \emph{classical parameters}, see Theorem 10.2.5 of  \cite{ionin-mohan}. It follows that for any odd $q$ there is a balanced weighing matrix with classical parameters.
\end{example}

A weighing matrix $W(n,p)$ is said to be in \emph{normal form}  if  {it} has the block configuration 
$$
\begin{bmatrix} \mathbf{0} & R \\ \mathbf{1} & D  \end{bmatrix},
$$
for some matrices $(0,\pm 1)$-matrices $R$ and $D$, where $\mathbf{0}$ and $\mathbf{1}$  are the $(n-p)\times 1$ column vector with all entries $0$ and the $p\times 1$ column vector of all entries $1$, respectively.  We call $R$ and $D$, respectively, the \emph{residual} and the \emph{derived} parts of the weighing matrix. Note that $RR^t = pI_{n-p}$, and $DD^t = pI_{p} - J_p$, and $RD^t = DR^t = 0$.  {By} permuting and negating some rows, if necessary, every weighing matrix can be assumed to be in normal form. \\

Two well-known results will be used throughout the paper. The first is a result related to the existence of balanced generalized weighing matrices \cite{ionin-mohan}.

\begin{prop}
For every prime power $p$ there is a $BGW(p+1,p,p-1)$ over a cyclic group $\mathbb{Z}_k$ for any $k$ dividing $p-1$.
\end{prop} 
 
 We introduce the following terminology.
 
\begin{definition}
A weighing matrix $W(p+1,p)$ is called a \emph{seed weighing} matrix, and a $BGW(p+1,p,p-1)$ over a cyclic group whose order divides $p-1$ is called a \emph{seed balanced generalized weighing} matrix.
\end{definition}

The use  {of a particular family of} \emph{orthogonal arrays} {whose existence is predicated on the existence of prime powers} is the key to the method presented in this paper. For details and definitions concerning orthogonal arrays see \cite{h-s-s-oa}.
\begin{prop}\label{oa}
For the prime power $p$ and the positive integer $m$,  there is an array ${\bf O}$ in $p$ symbols of  {dimensions} 
$p^{m+1}\times (\frac{p^{m+1}-1}{p-1})$  
for which any two distinct rows agree in $\frac{p^m-1}{p-1}$ columns. \end{prop}

\section{Weighing matrices}

The first result relates to weighing matrices $W(n,p)$ for $p$ a prime power.
\begin{theorem}\label{mw}
{If there is a $W(n,p)$, where $p$ is a prime power,} then there is a $$W\left(\frac{p^{m+1}-1}{p-1}(n-1)+1,p^{m+1}\right),$$ for each positive integer $m$.
\end{theorem}
\begin{proof}

Let $$W=\begin{bmatrix} \mathbf{0} & R \\ \mathbf{1} & D  \end{bmatrix}$$ {be the assumed $W(n,p)$.}
Let $\mathcal{W}$ be a $W\left(\frac{p^{m+1}-1}{p-1},p^{m}\right)$,  and consider {the matrix}
$$\mathcal{R}=\mathcal{W}\otimes {R}.$$ 
It is easy to see that the rows of the $(0,\pm 1)$-matrix {$\mathcal{R}$} 
are mutually orthogonal. Let $\bf {O}$ be the array in Proposition \ref{oa} and let $\mathcal{D}$ be the matrix obtained from $\bf {O}$ by replacing the $p$ symbols by the $p$ rows of the matrix $\mathbf{D}$ in any order. It follows that $$\mathcal{D}\mathcal{D}^t=p^{m+1}I_{p^{m+1}}-J_{p^{m+1}}.$$
The matrix $$\begin{bmatrix} \mathbf{0} & \mathcal{R} \\ \mathbf{1} & \mathcal {D}  \end{bmatrix}$$
is a $$W\left(\frac{p^{m+1}-1}{p-1}(n-1)+1,p^{m+1}\right)$$ for each positive integer $m$. 
\end{proof}
The existence of $W(43,25)$ is not known. In order to illustrate the construction above a $W(43,25)$ is constructed next.

\begin{example}
We Start with a $W(8,5)$
\[
W=
\left[\begin{array}{c|ccccccc}
1&1&1&1 & 1 & 0&0 &0\\
1&1&-&-  &  0 &1&0&0\\
1&-&1&-  & 0 &0 & 1&0\\
1&-&-&1  &0 &0 &0 &1\\
1 & 0&0 &0 &-&-&-&-\\\hline
0 &1&0&0&-&-&1&1 \\
0 &0 & 1&0&-&1&-&1\\
0 &0 &0 &1&-&1&1&-
\end{array}\right].
\]
Then 
$$
D=
\begin{bmatrix}
1&1&1&1&0&0&0\\
1&-&-&0&1&0&0 \\
-&1&-&0&0&1&0\\
-&-&1&0&0&0&1\\
0&0&0&-&-&-&-
\end{bmatrix},
$$
for which $DD^t = 5I - J$, and

$$
R = 
\begin{bmatrix}
1&0&0&-&-&1&1\\
0&1&0&-&1&-&1\\
0&0&1&-&1&1&-
\end{bmatrix},
$$
for which $RR^t = 5I$.

Replacing the five symbols in the $OA(5,2)$ shown in {the} Appendix with  the five rows of $R$ (in any order) the matrix  $\mathcal D$ is obtained. It is seen that $\D\D^t = 25{I}-J$.

If $W$ is any normal $W(6,5)$, then
\[
\begin{array}{cc}
\R = W \otimes R=
\begin{bmatrix}
0&1&1&1&1&1\\
1&0&1&-&-&1\\
1&1&0&1&-&-\\
1&-&1&0&1&-\\
1&-&-&1&0&1\\
1&1&-&-&1&0
\end{bmatrix}\otimes R=
&\begin{bmatrix}
0&R&R&R&R&R\\
R&0&R&\bar{R}&\bar{R}&R\\
R&R&0&R&\bar{R}&\bar{R}\\
R&\bar{R}&R&0&R&\bar{R}\\
R&\bar{R}&\bar{R}&R&0&R\\
R&R&\bar{R}&\bar{R}&R&0
\end{bmatrix}\end{array}\]
{where $\bar{R}=-R$,}. Then $\R\R^t = 25I$, and
moreover,  $\D\R^t = 0$. 
Consequently, the matrix
$$
X=
\left[\begin{array}{c|c}
\mathbf{1} & \D \\\hline \mathbf{0} & \R
\end{array}\right]
$$
forms a $W(43,25)$, which is shown in the Appendix.

\end{example}

The existence of weighing matrices of small {orders} has been verified through a number of papers. Most of {the} weighing matrices constructed {by} Theorem \ref{mw} are new. Table 1 below shows the orders of the constructed weighing matrices for small order.

We will use the superscripts $(v,k)^\dag$ and $(v,k)^\ddag$ to denote a new order/weight pair and balancedness over $\mathbb{Z}_2$, respectively.
\begin{longtable}[c]{| c | c |}
 \caption{Consequential order/weight pairs for $v \leq 1000$.\label{long}}\\

 \hline
Seed $(v,k)$ & Succident $(v',k')$ \\
\hline
 \endfirsthead

 \multicolumn{2}{c}{\emph{Continuation of Table \ref{long}}}\\
 \hline
 Seed $(v,k)$ & Succident $(v',k')$ \\
 \hline
 \endhead

 \hline
 \endfoot

$(7,4)$ & $(31,16)^\ddag$, $(127, 64)$, $(511, 256)$ \\ \hline
$(6,5)^\ddag$ & $(31, 25)^\ddag$, $(156, 125)^\ddag$, $(781, 625)^\ddag$ \\ \hline
$(8,5)$ & $(43, 25)^\dag$, $(218, 125)$ \\ \hline
$(10,5)$ & $(55, 25)$, $(280, 125)$ \\ \hline
$(12,5)$ & $(67, 25)$, $(342, 125)$ \\ \hline
$(8,7)^\ddag$ & $(57, 49)^\ddag$, $(400, 343)^\ddag$ \\ \hline
$(12,7)$ & $(89, 49)^\dag$, $(628, 343)$ \\ \hline
$(16,7)$ & $(121, 49)$, $(856, 343)$ \\ \hline
$(20,7)$ & $(153, 49)$ \\ \hline
$(10,8)$ & $(82, 64)$, $(658, 512)$ \\ \hline
$(12,8)$ & $(100, 64)$, $(804, 512)$ \\ \hline
$(14,8)$ & $(118, 64)$, $(950, 512)$ \\ \hline
$(16,3)$ & $(69,9)$, $(196,27)$, $(601, 81)$ \\ \hline
$(16,4)$ & $(76, 16)$, $(316, 64)$ \\ \hline
$(16,5)$ & $(91,25)$, $(466,125)$ \\ \hline
$(16,7)$ & $(121, 49)$, $(856, 343)$ \\ \hline
$(16,8)$ & $(136, 64)$ \\ \hline
$(16,11)$ & $(181, 121)^\dag$ \\ \hline
$(16,13)$ & $(211, 169)^\dag$ \\ \hline
$(10,9)^\ddag$ & $(91, 81)^\ddag$, $(820, 729)^\ddag$ \\ \hline
$(12,9)$ & $(111, 81)^\dag$ \\ \hline
$(13,9)$ & $(121,81)^\ddag$ \\ \hline
$(14,9)$ & $(131, 81)$ \\ \hline
$(16,9)$ & $(151, 81)$ \\ \hline
$(14,13)^\ddag$ & $(183, 169)^\ddag$ \\ \hline
$(16,13)$ & $(211, 169)^\dag$ \\ \hline
$(18,13)$ & $(239, 169)^\dag$ \\ \hline
$(19,9)^\ddag$ & $(181, 81)^\ddag$ \\ \hline
$(20,13)$ & $(267, 169)^\dag$
\end{longtable}

\section{An application of weighing matrices to symmetric designs}

\begin{definition} The balanced incomplete block design $BIBD(v,b,r,k,\lambda)$ is said to have a \emph{twin mate}  if the complement design is also a $BIBD(v,b,r,k,\lambda)$.
\end{definition}

\begin{example} The $BIBD(6,10,5,3,2)$ with the incidence matrix

\[
D=\left[
\begin{array}{cccccccccc}
 {1}&0& {1}& {1}& {1}&0&0&0& {1}&0\\
0& {1}&0& {1}& {1}& {1}&0&0&0& {1}\\
 {1}& {1}& {1}&0&0&0& {1}&0&0& {1}\\
 {1}&0&0& {1}&0& {1}& {1}& {1}&0&0\\
0& {1}&0&0& {1}&0& {1}& {1}& {1}&0\\
0&0& {1}&0&0& {1}&0& {1}& {1}& {1}
\end{array}
\right]
\] has a twin mate with the incidence matrix
  \[
C=\left[
\begin{array}{cccccccccc}
  {0}&  {1}&  {0}&  {0}&  {0}&  {1}&  {1}&  {1}&  {0}&  {1}\\
  {1}&  {0}&  {1}&  {0}&  {0}&  {0}&  {1}&  {1}&  {1}&  {0}\\
  {0}&  {0}&  {0}&  {1}&  {1}&  {1}&  {0}&  {1}&  {1}&  {0}\\
  {0}&  {1}&  {1}&  {0}&  {1}&  {0}&  {0}&  {0}&  {1}&  {1}\\
  {1}&  {0}&  {1}&  {1}&  {0}&  {1}&  {0}&  {0}&  {0}&  {1}\\
  {1}&  {1}&  {0}&  {1}&  {1}&  {0}&  {1}&  {0}&  {0}&  {0}
\end{array}
\right].\]
\end{example}
The construction method in the following theorem, perhaps surprisingly, is similar to the that of Theorem \ref{mw}. 

\begin{theorem}\label{md}
For the prime power $p$,  the residual design  $$R=BIBD\left(p+1,2p,p,\frac{p+1}{2},\frac{p-1}{2}\right) $$ and the derived design $$D=BIBD\left(p,2p,p-1,\frac{p-1}{2},\frac{p-3}{2}\right) $$ of a symmetric $$SBIBD(2p+1,p,\frac{p-1}{2})$$ are the seed designs  used to construct a symmetric $$SBIBD\left(2p(\frac{p^m-1}{p-1})+1,p^m,p^{m-1}(\frac{p-1}{2})\right)$$
for any positive integer $m$.
\end{theorem}
\begin{proof}
Let $W$ be a balanced $W(\frac{p^{m+1}}{p-1},p^m)$ and consider $\mathcal{R}=W\otimes R$. Let $W=W^+-W^-$, where $W^+$ and $W^-$ are  $(0,1)$-matrices. Let $R'$ be the twin mate to $R$, and consider $\mathcal{R}=W^+\otimes R+W^-\otimes R'$. It is routine, but tedious, to show that $\mathcal{R}$ is a quasi-residual 
$$BIBD\left((p+1)(\frac{p^{m+1}-1}{p-1}), 2p(\frac{p^{m+1}-1}{p-1}),p^{m+1},p^m(\frac{p+1}{2}),p^m(\frac{p-1}{2})\right).$$
Let $\bf {O}$ be the array in Proposition \ref{oa},  and let $\mathcal{D}$ be the matrix obtained from $\bf {O}$ by replacing the $p$ symbols by the $p$ rows of the matrix ${D}$ in any order. 
By carefully counting, it follows that $$\mathcal{D}\mathcal{D}^t=\frac{p^m(p+1)}{2}I_{p^{m+1}}+(\frac{p^m(p-1)}{2}-1)J_{p^{m+1}}.$$
The last step is to add the appropriate column  {so that}
\begin{align*}&\begin{bmatrix} \mathbf{0} & \mathcal{R} \\ \mathbf{1} & \mathcal {D}  \end{bmatrix}
\end{align*}
{is the required matrix.}
\end{proof}
\begin{remark}
The construction above {makes} use of the balanced weighing matrix $W$. Although the existence of such a balanced weighing matrix is known, by a method similar to the above construction and the use of the seed $W(p+1,p)$ weighing matrix, a balanced weighing matrix $W(\frac{p^{m+1}}{p-1},p^m)$ is reconstructed here.  The parameters of the design  {are} not new. See \cite{ionin-mohan} in which a generalized Hadamard matrix is used instead of an orthogonal array.
\end{remark}

\section{Balanced Generalized Weighing matrices}

One of the  {most} frequently used tools in the construction of many new symmetric designs  is the classical balanced generalized weighing matrices. We refer the reader to \cite{ionin-mohan} for the most comprehensive coverage of the subject. 

In this section, assuming the existence  {of seed} balanced generalized weighing matrices, the unifying technique used in previous sections is employed to offer a new construction for the classical BGW matrices.

\begin{theorem}\label{bgw}
For the prime power $p$, let $W$  {be a} $BGW(p+1,p,p-1)$ over a cyclic group $G$ of order  $p-1$. Then there is a $$BGW\left(\frac{p^{m+1}-1}{p-1}, p^m,p^{m-1}(p-1)\right)$$ over $G$ for any positive integer $m$.
\end{theorem}
\begin{proof}
Without loss of generality  {the} matrix $W$ is assumed to be in the normal form
$$W=\begin{bmatrix} \mathbf{0} & \mathbf{1}^t \\ \mathbf{1} & \mathcal {D}  \end{bmatrix}.$$
As in previous two cases, let $\bf {O}$ be the array in Proposition \ref{oa}, and let $\mathcal{D}$ be the matrix obtained from $\bf {O}$ by replacing the $p$ symbols by the $p$ rows of the matrix ${D}$ in any order. 
The $(0,G)$-matrix $\mathcal{D}$ {has dimensions}  $p^{m+1}\times p\left(\frac{p^{m+1}-1}{p-1}\right)$ with $p^{m+1}-1$ ones in each row. Noting that any two {distinct} rows of $\bf O$ have  exactly $\frac{p^m-1}{p-1}$ in the same column, there are $(\frac{p^m-1}{p-1})(p-1)=p^m-1$ identity elements of $G$ {upon taking the conjugate inner product between the rows}. The remaining $p^{m+1}$ columns in the two rows contain different rows of $W$ and so there are $p^{m}$ times of each element of the group except for the identity element which has one less. 
 Let 
$\mathcal{R}=W\otimes \mathbf{1}^t$.  Then 
{
$$\begin{bmatrix} \mathbf{0} & \mathcal{R} \\ \mathbf{1} & \mathcal{D}  \end{bmatrix}$$
}
is a  BGW$$\left(\frac{p^{m+1}-1}{p-1}, p^m,p^{m-1}(p-1)\right)$$ over $G$.
\end{proof}
\begin{example}
Let $p=3$ and consider 

\[
W=\begin{bmatrix} 0 & 1 & 1 &1\\1& 0 & 1 & -\\1 & - & 0 & 1\\1 & 1 & - & 0\end{bmatrix}.
\]
By applying Theorem \ref{bgw}, the balanced weighing matrices $$BGW({(3^{m+1}-1)/2},3^m, {2\cdot 3^{m-1}})$$ over $\mathbb{Z}_2$
{are} constructed.
\end{example}
 \begin{remark}
 The above construction is  {further} illuminated in an example for $p=5$ in the Appendix.
\end{remark}

\section{Association schemes}
The main reference for this section is \cite{bannai}.
Let $d$ be a positive integer. 
Let $X$ be a finite set and $R_i$ ($i\in\{0,1,\ldots,d\}$) a nonempty subset of $X\times X$. 

The \emph{adjacency matrix} $A_i$ of the graph with vertex set $X$ and edge set $R_i$ is a $(0,1)$-matrix indexed by the elements of $X$ such that $(A_i)_{xy}=1$ if $(x,y)\in R_i$, and $(A_i)_{xy}=0$ otherwise. 

A \emph{commutative association scheme} with $d$ classes is a pair $(X,\{R_i\}_{i=0}^d)$ satisfying the following:
\begin{enumerate}
\item $A_0=I_{|X|}$.
\item $\sum_{i=0}^d A_i = J_{|X|}$.
\item $A_i^t\in\{A_1,\ldots,A_d\}$ for any $i\in\{1,\ldots,d\}$.
\item For all $i$ and $j$, $A_i A_j$ is a linear combination of $A_0,A_1,\ldots,A_d$.
\item $A_i A_j=A_j A_i$ for all $i,j$. 
\end{enumerate}
We will also refer to $(0,1)$-matrices $A_0,A_1,\ldots,A_d$ satisfying 1-5 as a commutative association scheme. 

If an association scheme satisfies $A_i^t=A_i$,  for any $i\in\{1,\ldots,d\}$, then the association scheme is said to be \emph{symmetric}, in which case it is necessarily commutative.   

The vector space spanned by the collection $\{A_i\}$ forms a commutative algebra, denoted by $\mathcal{A}$, and is called the \emph{Bose-Mesner algebra}.

There exists a basis of $\mathcal{A}$ consisting of primitive idempotents, say, $E_0=(1/|X|)J_{|X|},E_1,\ldots,E_d$. 
Note that the collection $\{E_i\}$ are projections onto maximal common eigenspaces of $A_0,A_1,\ldots,A_d$.
Since  $\{A_0,A_1,\ldots,A_d\}$ and $\{E_0,E_1,\ldots,E_d\}$ are two bases in $\mathcal{A}$, there exist the change-of-bases matrices $P=(p_{ij})_{i,j=0}^d$, $Q=(q_{ij})_{i,j=0}^d$ so that
\begin{align*}
A_j=\sum_{i=0}^d p_{ij}E_i,\quad E_j=\frac{1}{|X|}\sum_{i=0}^d q_{ij}A_i.
\end{align*}
The matrices $P$ and $Q$ are said to be {\it the first and second eigenmatrices}, respectively.

Let $\mathbb{Z}_n=\{\bar{0},\bar{1},\ldots,\overline{n-1}\}$ be an additive cyclic group of order $n$, and let 
$W=(w_{ij})$ be a balanced generalized weighing matrix over $\mathbb{Z}_n$ with parameters $(v,k,\lambda)$.
Define the $v\times v$ $(0,1)$-matrices $B_{i}$ by  
\begin{align*}
W=\sum_{i=0}^{n-1} \bar{i} B_i \in \text{Mat}_v(\mathbb{C}[\mathbb{Z}_n]),
\end{align*}
where $\text{Mat}_v(\mathbb{C}[\mathbb{Z}_n])$ is the matrix algebra of $v \times v$ matrices over the group ring $\mathbb{C}[\mathbb{Z}_n]$.

Let $W^*$ denote the matrix with $(i,j)$-entry given by $-w_{j i}$ if $w_{ij}\in \mathbb{Z}_n$\textcolor{red},  and $0$ if $w_{i j}=0$.
If $W$ is a balanced generalized weighing matrix, then so is $W^*$.  
Thus, we have 
\begin{align*}
\sum_{i,j=0}^{n-1} \overline{i-j}B_{i} B_{j}^t=\sum_{i,j=1}^n \overline{i-j}B_{j}^t B_{i}=k I_{v}+\frac{\lambda}{n}\mathbb{Z}_n(J_v-I_v). 
\end{align*}    
Comparing entries with $\bar{h}$, for each $\bar{h}\in \mathbb{Z}_n$, yields
\begin{align}
\sum_{i=0}^{n-1} B_i B_i^t&=\sum_{i=0}^{n-1} B_i^t B_i=k I_{v}+\frac{\lambda}{n}(J_v-I_v),\label{eq:1}\\
\sum_{i=0}^{n-1} B_i B_{i-h}^t&=\sum_{i=0}^{n-1} B_{i-h}^t B_i=\frac{\lambda}{n}(J_v-I_v) \text{ for }h\neq 0.\label{eq:2}
\end{align}
Conversely, disjoint $(0,1)$-matrices $B_i$ ($i\in\{0,1,\ldots,n-1\}$) satisfying \eqref{eq:1} and \eqref{eq:2} give rise to a $BGW(v,k,\lambda)$ $W=\sum_{i=0}^{n-1} \bar{i} B_i$ over the cyclic group $\mathbb{Z}_n$ of order $n$. 

Let $P=\text{circ}(010\cdots0)$ be a cyclic matrix of order $n$.  
Define the adjacency matrices as follows:
\begin{align*}
A_{0,i}&=\begin{bmatrix}
P^i\otimes I_v&0\\
0& P^i\otimes I_v
\end{bmatrix},\quad (0\leq i \leq n-1) ,\\
A_1&=\begin{bmatrix}
J_n \otimes (J_v-I_v)&0\\
0&J_n \otimes (J_v-I_v)
\end{bmatrix},\\
A_{2,i}&=\begin{bmatrix}
0&\sum_{j=0}^{n-1}P^j\otimes B_{i+j}\\
\sum_{j=0}^{n-1}P^j \otimes B_{-i-j}^t&0
\end{bmatrix},\quad (0\leq i \leq n-1),\\
A_3&=\begin{bmatrix}
0&J_n\otimes(J_v-\sum_{j=0}^{n-1}B_j)\\
J_n\otimes(J_v-\sum_{j=0}^{n-1} B_j^t)&0
\end{bmatrix}.
\end{align*}

Note that when the weight $k$ is equal to the order $d$, namely, $W$ is a generalized Hadamard matrix, the matrix $A_3=0$.

\begin{theorem}\label{thm:ASBGW}
Let $W$ be a $BGW(v,k,\lambda;\mathbb{Z}_n)$. 
Then the following hold:
\begin{enumerate}
\item If $k=v$, then $\{A_{0,i},A_1,A_{2,i} \mid 0\leq i\leq n-1 \}$ is a commutative association scheme.
\item If $k<v$, then $\{A_{0,i},A_1,A_{2,i},A_3 \mid 0\leq i\leq n-1\}$ is a commutative association scheme.
\item The association schemes in $1,2$ are symmetric if and only if $n=2$. 
\end{enumerate}
\end{theorem}
\begin{proof}
1, 2: 
It is easily checked that $A_{0,0}$ is the identity matrix,  the sum of all adjacency matrices is equal to the all-ones matrix and, $A_{0,i}^t=A_{0,-i}$ for $0\leq i\leq n-1$, $A_1^t=A_1$, $A_{2,i}^t=A_{2,-i}$ and $A_3^t=A_3$, where the second index $-i$ is taken modulo $n$.
Set $\mathcal{A}=\text{span}_{\mathbb{C}}\{A_{0,i},A_1,A_{2,i},A_3 \mid 0\leq i \leq n-1\}$.
It is enough to show that $\mathcal{A}$ is closed under multiplication and is commutative.

For $0\leq i,j\leq n-1$, it is easy to see that 
\begin{align*}
A_{0,i} A_{0,j}&=A_{0,j} A_{0,i}=A_{0,i+j}\in\mathcal{A},\\
A_{0,i} A_1&=A_1 A_{0,i},
=A_1\in\mathcal{A}, \\
A_1^2&
=n(v-1)\sum_{i=0}^{n-1}A_{0,i} +n(v-2)A_1\in\mathcal{A}.
\end{align*}

For $0\leq i,j\leq n-1$, 
\begin{align*}
A_{0,i} A_{2,j}&=A_{2,j} A_{0,i}=\begin{bmatrix}
0&\sum_{k=0}^{n-1} P^{i+k}\otimes B_{j+k}\\
\sum_{k=0}^{n-1} P^{i+k}\otimes B_{-j-k}^t &0
\end{bmatrix}\\
&=\begin{bmatrix}
0&\sum_{k=0}^{n-1} P^{k}\otimes B_{j-i+k}\\
\sum_{k=0}^{n-1} P^{k}\otimes B_{-j+i-k}^t &0
\end{bmatrix}\\
&=A_{2,-i+j}\in \mathcal{A}. 
\end{align*}

For $0\leq i\leq n-1$, 
\begin{align*}
A_1 A_{2,i}&=A_{2,i}A_1 =\begin{bmatrix}
0&J_n\otimes(k J_v-\sum_{k=0}^{n-1}B_k)\\
J_n\otimes(k J_v-\sum_{k=0}^{n-1}B_k^t)&0
\end{bmatrix}\\
&=(k-1)\sum_{j=0}^{n-1} A_{2,j}+k A_3\in \mathcal{A}.
\end{align*}

For $0\leq i,j \leq n-1$, we use the following equality:
\begin{align*}
\sum_{h,\ell=0}^{n-1}P^{h+\ell}\otimes B_{i+h}B_{-j-\ell}^t
&=\sum_{\alpha,h=0}^{n-1} P^\alpha \otimes B_{i+h}B_{-j-(\alpha-h)}^t \\
&=\sum_{\alpha,\beta=0}^{n-1} P^\alpha \otimes B_\beta B_{-j-(\alpha-(\beta-i))}^t \\
&=\sum_{\alpha,\beta=0}^{n-1} P^\alpha \otimes B_\beta B_{-j-i-\alpha+\beta}^t \\
&=\sum_{\alpha=0}^{n-1} P^\alpha \otimes (\sum_{\beta=0}^{n-1} B_\beta B_{-j-i-\alpha+\beta}^t) \\
&=P^{i+j} \otimes (\sum_{\beta=0}^{n-1} B_\beta B_\beta^t)+\sum_{\substack{0\leq \alpha\leq n-1\\ \alpha\neq i+j}} P^\alpha \otimes (\sum_{\beta=0}^{n-1} B_\beta B_{-j-i-\alpha+\beta}^t) \\
&=P^{i+j} \otimes (k I_v+\frac{\lambda}{n}(J_v-I_v))+\sum_{\substack{0\leq \alpha\leq n-1\\ \alpha\neq i+j}} P^\alpha \otimes \frac{\lambda}{n}(J_v-I_v) \\
&=k P^{i+j} \otimes I_v+\frac{\lambda}{n} \sum_{\alpha=0}^{n-1} P^\alpha \otimes (J_v-I_v) \\
&=k P^{i+j} \otimes I_v+\frac{\lambda}{n} J_n \otimes (J_v-I_v).
\end{align*}
Similarly, we have 
\begin{align*}
\sum_{h,\ell=0}^{n-1}P^{h+\ell}\otimes B_{-i-h}^t B_{j+\ell}
=k P^{i+j} \otimes I_v+\frac{\lambda}{n} J_n \otimes (J_v-I_v).
\end{align*}
Thus, 
\begin{align*}
A_{2,i}A_{2,j}&=A_{2,j}A_{2,i}=\begin{bmatrix}
\sum_{h,\ell=0}^{n-1}P^{h+\ell}\otimes B_{i+h}B_{j-\ell}^t  &0\\
0&\sum_{h,\ell=0}^{n-1}P^{h+\ell}\otimes B_{i-h}^t B_{j+\ell}
\end{bmatrix}\\
&=k\begin{bmatrix}
P^{i+j} \otimes I_v  &0\\
0 & P^{i+j} \otimes I_v
\end{bmatrix}
+\frac{\lambda}{n}\begin{bmatrix}
J_n\otimes (J_v-I_v)  &0\\
0&J_n\otimes (J_v-I_v)
\end{bmatrix}\\
&=k A_{0,i+j}+\frac{\lambda}{n}A_1 \in \mathcal{A}. 
\end{align*}

Finally, we have the following equalities:
\begin{align*}
A&_{0,i}A_3=A_3 A_{0,i}=A_3,\\
A&_1 A_2=A_2A_1
=(v-k)\sum_{j=0}^{n-1}A_{2,j}+(v-k-1)A_3,\\
A&_{2,i}A_3=A_3A_{2,i}
=(k-\lambda)A_1,\\
A_3A_3&=
n(v-k)\sum_{j=0}^{n-1} A_{0,j}+n(v-2k+\lambda)A_1,
\end{align*}
each of which belongs to $\mathcal{A}$.
Thus, the vector space $\mathcal{A}$ is closed under multiplication and is commutative.

3: The association scheme is symmetric if and only if  $A_{0,i}$ and $A_{2,i}$ are symmetric, for each $i=1,\ldots,n-1$.
The latter condition is equivalent to $x=-x$ for each $g\in \mathbb{Z}_n\setminus \{\bar{0}\}$, that is $n$ is $2$.
\end{proof}

Set 
\begin{align*}
E_0&=\frac{1}{2nv}(\sum_{j=0}^{n-1}(A_{0,j}+A_{2,j})+A_1+A_3),\displaybreak[0]\\
E_1&=\frac{1}{2nv}(\sum_{j=0}^{n-1}(A_{0,j}-A_{2,j})+A_1-A_3),\displaybreak[0]\\
E_{2,1}&=\frac{1}{2nv}(\sum_{j=0}^{n-1}((v-1)A_{0,j}+\frac{\sqrt{(v-1)(v-k)}}{\sqrt{k}}A_{2,j})-A_1-\frac{\sqrt{(v-1)k}}{\sqrt{v-k}}A_3),\displaybreak[0]\\
E_{2,2}&=\frac{1}{2nv}(\sum_{j=0}^{n-1}((v-1)A_{0,j}-\frac{\sqrt{(v-1)(v-k)}}{\sqrt{k}}A_{2,j})-A_1+\frac{\sqrt{(v-1)k}}{\sqrt{v-k}}A_3),\displaybreak[0]\\
E_{3,i}&=\frac{1}{2nv}\sum_{j=0}^{n-1}w^{ij}(vA_{0,j}+\frac{v}{\sqrt{k}}A_{2,j}),\quad (i\in\{1,\ldots,n-1\}),\\
E_{4,i}&=\frac{1}{2nv}\sum_{j=0}^{n-1}w^{ij}(vA_{0,j}-\frac{v}{\sqrt{k}}A_{2,j}),\quad (i\in\{1,\ldots,n-1\}),
\end{align*}
where $w=e^{2\pi\sqrt{-1}/n}$. 
It readily follows from the intersection numbers of the association schemes that 
\begin{itemize}
\item $E_0,E_1,E_{2,1},E_{2,2},E_{3,i},E_{4,i}$ $(i\in\{1,\ldots,n-1\})$ are the primitive idempotents if $v>k$, and 
\item $E_0,E_1,E_2=E_{2,1}+E_{2,2},E_{3,i},E_{4,i}$ $(i\in\{1,\ldots,n-1\})$ are the primitive idempotents if $v=k$. 
\end{itemize}
Therefore, we have the following. 
\begin{theorem}\label{thm:eigen}
The eigenmatrices $P,Q$ of the association schemes in Theorem~\ref{thm:ASBGW} are given as follows. 
\begin{enumerate}
\item If $k=v$, then 
\begin{align*}
P&=\bbordermatrix{
      & A_{0,i} & A_1 & A_{2,i}   \cr
E_0 & 1 & n(v-1)  & k   \cr
E_1 & 1 & n(v-1) &  -k  \cr
E_2 & 1 & -n &  0  \cr
E_{3,j} & w^{-ij} & 0 &  \sqrt{k}w^{-ij}  \cr
E_{4,j} & w^{-ij} & 0 & -\sqrt{k}w^{-ij}  
},\\
Q&=\bbordermatrix{
      & E_{0} & E_1 & E_2 & E_{3,j} & E_{4,j}  \cr
A_{0,i} & 1 & 1 & 2(v-1)  & v w^{ij} & v w^{ij} \cr
A_1    & 1 & 1 & -2 & 0 & 0 \cr
A_{2,i} & 1 & -1 & 0 & \frac{v}{\sqrt{k}}w^{ij} & -\frac{v}{\sqrt{k}}w^{ij}  
}.
\end{align*}
\item If $k<v$, then 
\begin{align*}
P&=\bbordermatrix{
      & A_{0,i} & A_1 & A_{2,i}  & A_{3}  \cr
E_0 & 1 & n(v-1)  & k & n(v-k) \cr
E_1 & 1 & n(v-1) &  -k & -n(v-k) \cr
E_{2,1} & 1 & -n & \sqrt{\frac{k(v-k)}{v-1}}  & -n\sqrt{\frac{k(v-k)}{v-1}} \cr
E_{2,2} & 1 & -n & -\sqrt{\frac{k(v-k)}{v-1}}  & n\sqrt{\frac{k(v-k)}{v-1}} \cr
E_{3,j} & w^{-ij} & 0 &  \sqrt{k}w^{-ij} & 0 \cr
E_{4,j} & w^{-ij} & 0 & -\sqrt{k}w^{-ij} & 0
},\displaybreak[0]\\
Q&=\bbordermatrix{
      & E_{0} & E_1 & E_{2,1}  & E_{2,2} & E_{3,j} & E_{4,j}  \cr
A_{0,i} & 1 & 1 & v-1 & v-1 & v w^{ij} & v w^{ij} \cr
A_1    & 1 & 1 & -1 & -1 & 0 & 0 \cr
A_{2,i} & 1 & -1 & \sqrt{\frac{(v-1)(v-k)}{k}} & -\sqrt{\frac{(v-1)(v-k)}{k}} & \frac{v}{\sqrt{k}}w^{ij} & -\frac{v}{\sqrt{k}}w^{ij} \cr
A_3    & 1 & -1 & -\sqrt{\frac{(v-1)k}{v-k}} & \sqrt{\frac{(v-1)k}{v-k}} & 0 & 0  
}.
\end{align*}
\end{enumerate}
\end{theorem}

We now show the converse of the theorem, that is, a BGW over a cyclic group can be constructed by the association scheme as follows. 
\begin{theorem}
If there exists a commutative association scheme with the same eigenmatrices in Theorem~\ref{thm:eigen}, then there exists a $BGW(v,k,\lambda)$ over $\mathbb{Z}_n$. 
\end{theorem}
\begin{proof}
We deal with 1 and 2 simultaneously. 

Since $\sum_{i=0}^{n-1}A_{0,i}+A_1$ has the eigenvalues $nv,0$ with multiplicities $2,2nv-2$, respectively. 
Therefore, after suitably permuting the vertices, we may assume that 
$$
\sum_{i=0}^{n-1}A_{0,i}+A_1=\begin{bmatrix} J_{nv} & 0 \\ 0 & J_{nv} \end{bmatrix}.   
$$
Similarly $\sum_{i=0}^{n-1}A_{0,i}$ has the eigenvalues $n,0$ with multiplicities $2v,2nv-2v$, respectively. 
Then we may assume that 
$$
\sum_{i=0}^{n-1}A_{0,i}=\begin{bmatrix} I_v\otimes J_n & 0 \\ 0 & I_v\otimes J_n \end{bmatrix}.   
$$

The set $\{A_{0,0},A_{0,1},\ldots,A_{0,n-1}\}$ forms a cyclic group of order $n$ and the matrix $A_{0,1}$ generates the group. 
Then $A_{0,1}$ is of the form 
$$\begin{bmatrix}U_1 & 0 & \cdots & 0 \\ 0 & U_2 & \cdots  & 0 \\ \vdots & \vdots & \ddots  & 0 \\0 & 0 & \cdots  & U_{2v} \end{bmatrix}, $$ where each $U_i$ is a permutation matrix of order $n$. 
Then, after suitably changing the ordering of the vertices, we may assume that $U_i=P:=\text{circ}(010\cdots0)$ for each $i$.  
Therefore $A_{0,1}=\begin{bmatrix} I_v\otimes P&  0 \\ 0 & I_v\otimes P\end{bmatrix}$. 

By considering the map $(x,y)\mapsto (y,x)$,  we have $A_{0,1}=\begin{bmatrix} P\otimes I_v & 0 \\ 0 & P\otimes I_v\end{bmatrix}$.  
Then it is shown that $A_{0,i}$ ($i\in\{0,1,\ldots,n-1\}$) and $A_1$ have the desired form. 

By the intersection numbers, we have $A_{0,1}^{-1}A_{2,0}A_{0,1}=A_{2,0}$, which shows that $A_{2,0}$ is of the desired for some $(0,1)$-matrices $B_0,B_1,\ldots,B_{n-1}$. 
Again, by the intersection numbers, we have $A_{0,1}A_{2,j}=A_{2,j-1}$, which implies that $A_{2,j}$ is of the desired form. 

Finally, by the intersection numbers, $A_{2,i}A_{2,j}=kA_{0,i+j}+\frac{\lambda}{n}A_1$ where $\lambda=\frac{k(k-1)}{v-1}$. 
This shows that \eqref{eq:1} and \eqref{eq:2} hold for $B_0,B_1,\ldots,B_{n-1}$. 
Therefore, $W=\sum_{k=0}^{n-1}\bar{k}B_k$ is a BGW over the cyclic group, where $\bar{k}$ is an element of $\mathbb{Z}_n$. 
\end{proof}

\begin{remark}
 It is important to remember that the results of this section are predicated upon the group in question being cyclic. This leaves out the case of those generalized Hadamard matrices over the elementary abelian groups whose orders are non-prime. Though we do not explicitly show it here, one may replace the circulant matrices $\{P^i \mid 0 \leq i < n\}$ with the collection $\{U_g \mid g \in G\}$, where $G$ is a finite abelian group, and where $U_g$ is a regular linear representation of the group element $g$. This still omits the sporadic cases of BGW matrices over non-abelian groups, but all of the known infinite families are now encapsulated by the result.
\end{remark}

\section*{Acknowledgments.}
Hadi Kharaghani is supported by the Natural Sciences and 
Engineering  Research Council of Canada (NSERC).  Sho Suda is supported by JSPS KAKENHI Grant Number 18K03395.

\begin{rezabib}

\bib{bannai}{book}{
    author = {Eiichi Bannai  and Tatsuro Ito},
     TITLE = {Algebraic combinatorics. {I}},
      NOTE = {Association schemes},
 PUBLISHER = {The Benjamin/Cummings Publishing Co., Inc., Menlo Park, CA},
      YEAR = {1984},
     PAGES = {xxiv+425},
      ISBN = {0-8053-0490-8},
}

\bib{Khar}{article}{
   author={Hadi Kharaghani, },
   title={New class of weighing matrices},
   journal={Ars Combin.},
   volume={19},
   date={1985},
   pages={69--72},
   issn={0381-7032},
}

\bib{se-yam}{book}{
    author = {Jennifer Seberry  and Mieko Yamada},
     TITLE = {Hadamard matrices, sequences, and block designs},
 BOOKTITLE = {Contemporary design theory},
    SERIES = {Wiley-Intersci. Ser. Discrete Math. Optim.},
     PAGES = {431--560},
 PUBLISHER = {Wiley, New York},
      YEAR = {1992},
      }
\bib{Seb}{book}{
   author={Seberry, Jennifer},
   title={Orthogonal designs},
   note={Hadamard matrices, quadratic forms and algebras;
   Revised and updated edition of the 1979 original [MR0534614]},
   publisher={Springer, Cham},
   date={2017},
   pages={xxiii+453},
   isbn={978-3-319-59031-8},
   isbn={978-3-319-59032-5},
   doi={10.1007/978-3-319-59032-5},
}
\bib{h-s-s-oa}{book}{
    author = {A. S. Hedayat and N. J. A. Sloane and John Stufken},
     title = {Orthogonal arrays},
    series= {Springer Series in Statistics},
      note = {Theory and applications,
              With a foreword by C. R. Rao},
 publisher = {Springer-Verlag, New York},
      date = {1999},
     pages = {xxiv+416},
      isbn= {0-387-98766-5},
   }
   
\bib{ionin-mohan}{book}{
    author = {Yury J. Ionin and Mohan S. Shrikhande},
     TITLE = {Combinatorics of symmetric designs},
    SERIES = {New Mathematical Monographs},
    VOLUME = {5},
 PUBLISHER = {Cambridge University Press, Cambridge},
      YEAR = {2006},
     PAGES = {xiv+520},
      ISBN = {978-0-521-81833-9; 0-521-81833-8},
}

\end{rezabib}

\section*{Appendix}

We begin with the orthogonal array of dimension $25 \times 6$ over the alphabet $\{1, \dots, 5\}$ shown transposed below.
\[
\left[
\begin{array}{ccccccccccccccccccccccccc}
1&1&1&1&1&2&2&2&2&2&3&3&3&3&3&4&4&4&4&4&5&5&5&5&5\\
1&2&3&4&5&1&2&3&4&5&1&3&5&2&4&1&4&2&5&3&1&5&4&3&2\\
1&2&3&4&5&2&3&4&5&1&3&5&2&4&1&4&2&5&3&1&5&4&3&2&1\\
1&2&3&4&5&3&4&5&1&2&5&2&4&1&3&2&5&3&1&4&4&3&2&1&5\\
1&2&3&4&5&4&5&1&2&3&2&4&1&3&5&5&3&1&4&2&3&2&1&5&4\\
1&2&3&4&5&5&1&2&3&4&4&1&3&5&2&3&1&4&2&5&2&1&5&4&3\\
\end{array}
\right].
\]
Using this array, we construct a $W(43,25)$ as
\begin{scriptsize}
\[
\arraycolsep=1.0pt\def\arraystretch{1.0}
\left[
\begin{array}{ccccccccccccccccccccccccccccccccccccccccccc}
0&0&0&0&0&0&0&0&0&0&1&-&1&1&-&0&0&1&-&1&1&-&0&0&1&-&1&1&-&0&0&1&-&1&1&-&0&0&1&-&1&1&-\\
0&0&0&0&0&0&0&0&0&1&0&-&-&1&1&0&1&0&-&-&1&1&0&1&0&-&-&1&1&0&1&0&-&-&1&1&0&1&0&-&-&1&1\\
0&0&0&0&0&0&0&0&1&0&0&-&1&-&1&1&0&0&-&1&-&1&1&0&0&-&1&-&1&1&0&0&-&1&-&1&1&0&0&-&1&-&1\\
0&0&0&1&-&1&1&-&0&0&0&0&0&0&0&0&0&1&-&1&1&-&0&0&-&1&-&-&1&0&0&-&1&-&-&1&0&0&1&-&1&1&-\\
0&0&1&0&-&-&1&1&0&0&0&0&0&0&0&0&1&0&-&-&1&1&0&-&0&1&1&-&-&0&-&0&1&1&-&-&0&1&0&-&-&1&1\\
0&1&0&0&-&1&-&1&0&0&0&0&0&0&0&1&0&0&-&1&-&1&-&0&0&1&-&1&-&-&0&0&1&-&1&-&1&0&0&-&1&-&1\\
0&0&0&1&-&1&1&-&0&0&1&-&1&1&-&0&0&0&0&0&0&0&0&0&1&-&1&1&-&0&0&-&1&-&-&1&0&0&-&1&-&-&1\\
0&0&1&0&-&-&1&1&0&1&0&-&-&1&1&0&0&0&0&0&0&0&0&1&0&-&-&1&1&0&-&0&1&1&-&-&0&-&0&1&1&-&-\\
0&1&0&0&-&1&-&1&1&0&0&-&1&-&1&0&0&0&0&0&0&0&1&0&0&-&1&-&1&-&0&0&1&-&1&-&-&0&0&1&-&1&-\\
0&0&0&1&-&1&1&-&0&0&-&1&-&-&1&0&0&1&-&1&1&-&0&0&0&0&0&0&0&0&0&1&-&1&1&-&0&0&-&1&-&-&1\\
0&0&1&0&-&-&1&1&0&-&0&1&1&-&-&0&1&0&-&-&1&1&0&0&0&0&0&0&0&0&1&0&-&-&1&1&0&-&0&1&1&-&-\\
0&1&0&0&-&1&-&1&-&0&0&1&-&1&-&1&0&0&-&1&-&1&0&0&0&0&0&0&0&1&0&0&-&1&-&1&-&0&0&1&-&1&-\\
0&0&0&1&-&1&1&-&0&0&-&1&-&-&1&0&0&-&1&-&-&1&0&0&1&-&1&1&-&0&0&0&0&0&0&0&0&0&1&-&1&1&-\\
0&0&1&0&-&-&1&1&0&-&0&1&1&-&-&0&-&0&1&1&-&-&0&1&0&-&-&1&1&0&0&0&0&0&0&0&0&1&0&-&-&1&1\\
0&1&0&0&-&1&-&1&-&0&0&1&-&1&-&-&0&0&1&-&1&-&1&0&0&-&1&-&1&0&0&0&0&0&0&0&1&0&0&-&1&-&1\\
0&0&0&1&-&1&1&-&0&0&1&-&1&1&-&0&0&-&1&-&-&1&0&0&-&1&-&-&1&0&0&1&-&1&1&-&0&0&0&0&0&0&0\\
0&0&1&0&-&-&1&1&0&1&0&-&-&1&1&0&-&0&1&1&-&-&0&-&0&1&1&-&-&0&1&0&-&-&1&1&0&0&0&0&0&0&0\\
0&1&0&0&-&1&-&1&1&0&0&-&1&-&1&-&0&0&1&-&1&-&-&0&0&1&-&1&-&1&0&0&-&1&-&1&0&0&0&0&0&0&0\\
1&0&0&0&-&-&-&-&0&0&0&-&-&-&-&0&0&0&-&-&-&-&0&0&0&-&-&-&-&0&0&0&-&-&-&-&0&0&0&-&-&-&-\\
1&0&0&0&-&-&-&-&-&-&1&0&0&0&1&-&-&1&0&0&0&1&-&-&1&0&0&0&1&-&-&1&0&0&0&1&-&-&1&0&0&0&1\\
1&0&0&0&-&-&-&-&1&-&-&0&0&1&0&1&-&-&0&0&1&0&1&-&-&0&0&1&0&1&-&-&0&0&1&0&1&-&-&0&0&1&0\\
1&0&0&0&-&-&-&-&-&1&-&0&1&0&0&-&1&-&0&1&0&0&-&1&-&0&1&0&0&-&1&-&0&1&0&0&-&1&-&0&1&0&0\\
1&0&0&0&-&-&-&-&1&1&1&1&0&0&0&1&1&1&1&0&0&0&1&1&1&1&0&0&0&1&1&1&1&0&0&0&1&1&1&1&0&0&0\\
1&-&-&1&0&0&0&1&0&0&0&-&-&-&-&-&-&1&0&0&0&1&1&-&-&0&0&1&0&-&1&-&0&1&0&0&1&1&1&1&0&0&0\\
1&-&-&1&0&0&0&1&-&-&1&0&0&0&1&1&-&-&0&0&1&0&-&1&-&0&1&0&0&1&1&1&1&0&0&0&0&0&0&-&-&-&-\\
1&-&-&1&0&0&0&1&1&-&-&0&0&1&0&-&1&-&0&1&0&0&1&1&1&1&0&0&0&0&0&0&-&-&-&-&-&-&1&0&0&0&1\\
1&-&-&1&0&0&0&1&-&1&-&0&1&0&0&1&1&1&1&0&0&0&0&0&0&-&-&-&-&-&-&1&0&0&0&1&1&-&-&0&0&1&0\\
1&-&-&1&0&0&0&1&1&1&1&1&0&0&0&0&0&0&-&-&-&-&-&-&1&0&0&0&1&1&-&-&0&0&1&0&-&1&-&0&1&0&0\\
1&1&-&-&0&0&1&0&0&0&0&-&-&-&-&1&-&-&0&0&1&0&1&1&1&1&0&0&0&-&-&1&0&0&0&1&-&1&-&0&1&0&0\\
1&1&-&-&0&0&1&0&1&-&-&0&0&1&0&1&1&1&1&0&0&0&-&-&1&0&0&0&1&-&1&-&0&1&0&0&0&0&0&-&-&-&-\\
1&1&-&-&0&0&1&0&1&1&1&1&0&0&0&-&-&1&0&0&0&1&-&1&-&0&1&0&0&0&0&0&-&-&-&-&1&-&-&0&0&1&0\\
1&1&-&-&0&0&1&0&-&-&1&0&0&0&1&-&1&-&0&1&0&0&0&0&0&-&-&-&-&1&-&-&0&0&1&0&1&1&1&1&0&0&0\\
1&1&-&-&0&0&1&0&-&1&-&0&1&0&0&0&0&0&-&-&-&-&1&-&-&0&0&1&0&1&1&1&1&0&0&0&-&-&1&0&0&0&1\\
1&-&1&-&0&1&0&0&0&0&0&-&-&-&-&-&1&-&0&1&0&0&-&-&1&0&0&0&1&1&1&1&1&0&0&0&1&-&-&0&0&1&0\\
1&-&1&-&0&1&0&0&-&1&-&0&1&0&0&-&-&1&0&0&0&1&1&1&1&1&0&0&0&1&-&-&0&0&1&0&0&0&0&-&-&-&-\\
1&-&1&-&0&1&0&0&-&-&1&0&0&0&1&1&1&1&1&0&0&0&1&-&-&0&0&1&0&0&0&0&-&-&-&-&-&1&-&0&1&0&0\\
1&-&1&-&0&1&0&0&1&1&1&1&0&0&0&1&-&-&0&0&1&0&0&0&0&-&-&-&-&-&1&-&0&1&0&0&-&-&1&0&0&0&1\\
1&-&1&-&0&1&0&0&1&-&-&0&0&1&0&0&0&0&-&-&-&-&-&1&-&0&1&0&0&-&-&1&0&0&0&1&1&1&1&1&0&0&0\\
1&1&1&1&1&0&0&0&0&0&0&-&-&-&-&1&1&1&1&0&0&0&-&1&-&0&1&0&0&1&-&-&0&0&1&0&-&-&1&0&0&0&1\\
1&1&1&1&1&0&0&0&1&1&1&1&0&0&0&-&1&-&0&1&0&0&1&-&-&0&0&1&0&-&-&1&0&0&0&1&0&0&0&-&-&-&-\\
1&1&1&1&1&0&0&0&-&1&-&0&1&0&0&1&-&-&0&0&1&0&-&-&1&0&0&0&1&0&0&0&-&-&-&-&1&1&1&1&0&0&0\\
1&1&1&1&1&0&0&0&1&-&-&0&0&1&0&-&-&1&0&0&0&1&0&0&0&-&-&-&-&1&1&1&1&0&0&0&-&1&-&0&1&0&0\\
1&1&1&1&1&0&0&0&-&-&1&0&0&0&1&0&0&0&-&-&-&-&1&1&1&1&0&0&0&-&1&-&0&1&0&0&1&-&-&0&0&1&0\\
\end{array}
\right]
\]
\end{scriptsize}

By a similar process, a $BGW(30,25,20; \mathbb{Z}_4)$ is constructed as
\[
\arraycolsep=1.0pt\def\arraystretch{1.0}
\left[
\begin{array}{ccccccccccccccccccccccccccccccc}
0&0&0&0&0&0&4&4&4&4&4&4&4&4&4&4&4&4&4&4&4&4&4&4&4&4&4&4&4&4&4\\
0&4&4&4&4&4&0&0&0&0&0&3&3&3&3&3&4&4&4&4&4&1&1&1&1&1&2&2&2&2&2\\
0&4&4&4&4&4&1&1&1&1&1&0&0&0&0&0&3&3&3&3&3&2&2&2&2&2&4&4&4&4&4\\
0&4&4&4&4&4&2&2&2&2&2&1&1&1&1&1&0&0&0&0&0&4&4&4&4&4&3&3&3&3&3\\
0&4&4&4&4&4&3&3&3&3&3&4&4&4&4&4&2&2&2&2&2&0&0&0&0&0&1&1&1&1&1\\
0&4&4&4&4&4&4&4&4&4&4&2&2&2&2&2&1&1&1&1&1&3&3&3&3&3&0&0&0&0&0\\
4&0&3&4&1&2&0&3&4&1&2&0&3&4&1&2&0&3&4&1&2&0&3&4&1&2&0&3&4&1&2\\
4&0&3&4&1&2&1&0&3&2&4&1&0&3&2&4&1&0&3&2&4&1&0&3&2&4&1&0&3&2&4\\
4&0&3&4&1&2&2&1&0&4&3&2&1&0&4&3&2&1&0&4&3&2&1&0&4&3&2&1&0&4&3\\
4&0&3&4&1&2&3&4&2&0&1&3&4&2&0&1&3&4&2&0&1&3&4&2&0&1&3&4&2&0&1\\
4&0&3&4&1&2&4&2&1&3&0&4&2&1&3&0&4&2&1&3&0&4&2&1&3&0&4&2&1&3&0\\
4&1&0&3&2&4&0&3&4&1&2&1&0&3&2&4&2&1&0&4&3&3&4&2&0&1&4&2&1&3&0\\
4&1&0&3&2&4&1&0&3&2&4&2&1&0&4&3&3&4&2&0&1&4&2&1&3&0&0&3&4&1&2\\
4&1&0&3&2&4&2&1&0&4&3&3&4&2&0&1&4&2&1&3&0&0&3&4&1&2&1&0&3&2&4\\
4&1&0&3&2&4&3&4&2&0&1&4&2&1&3&0&0&3&4&1&2&1&0&3&2&4&2&1&0&4&3\\
4&1&0&3&2&4&4&2&1&3&0&0&3&4&1&2&1&0&3&2&4&2&1&0&4&3&3&4&2&0&1\\
4&2&1&0&4&3&0&3&4&1&2&2&1&0&4&3&4&2&1&3&0&1&0&3&2&4&3&4&2&0&1\\
4&2&1&0&4&3&2&1&0&4&3&4&2&1&3&0&1&0&3&2&4&3&4&2&0&1&0&3&4&1&2\\
4&2&1&0&4&3&4&2&1&3&0&1&0&3&2&4&3&4&2&0&1&0&3&4&1&2&2&1&0&4&3\\
4&2&1&0&4&3&1&0&3&2&4&3&4&2&0&1&0&3&4&1&2&2&1&0&4&3&4&2&1&3&0\\
4&2&1&0&4&3&3&4&2&0&1&0&3&4&1&2&2&1&0&4&3&4&2&1&3&0&1&0&3&2&4\\
4&3&4&2&0&1&0&3&4&1&2&3&4&2&0&1&1&0&3&2&4&4&2&1&3&0&2&1&0&4&3\\
4&3&4&2&0&1&3&4&2&0&1&1&0&3&2&4&4&2&1&3&0&2&1&0&4&3&0&3&4&1&2\\
4&3&4&2&0&1&1&0&3&2&4&4&2&1&3&0&2&1&0&4&3&0&3&4&1&2&3&4&2&0&1\\
4&3&4&2&0&1&4&2&1&3&0&2&1&0&4&3&0&3&4&1&2&3&4&2&0&1&1&0&3&2&4\\
4&3&4&2&0&1&2&1&0&4&3&0&3&4&1&2&3&4&2&0&1&1&0&3&2&4&4&2&1&3&0\\
4&4&2&1&3&0&0&3&4&1&2&4&2&1&3&0&3&4&2&0&1&2&1&0&4&3&1&0&3&2&4\\
4&4&2&1&3&0&4&2&1&3&0&3&4&2&0&1&2&1&0&4&3&1&0&3&2&4&0&3&4&1&2\\
4&4&2&1&3&0&3&4&2&0&1&2&1&0&4&3&1&0&3&2&4&0&3&4&1&2&4&2&1&3&0\\
4&4&2&1&3&0&2&1&0&4&3&1&0&3&2&4&0&3&4&1&2&4&2&1&3&0&3&4&2&0&1\\
4&4&2&1&3&0&1&0&3&2&4&0&3&4&1&2&4&2&1&3&0&3&4&2&0&1&2&1&0&4&3\\
\end{array}
\right].
\]
\end{document}